\documentclass[journal]{IEEEtran}

\ifCLASSINFOpdf
\else
   \usepackage[dvips]{graphicx}
\fi
\usepackage{url}

\usepackage{graphicx}

\usepackage{color,xcolor}

\usepackage{amsthm}
\newtheorem{thm}{Theorem}
\newtheorem{lem}{Lemma}

\newtheorem{remk}{Remark}

\usepackage{amsmath,amsfonts}

\usepackage{algorithmic}
\usepackage{algorithm}

\usepackage[caption=false,font=normalsize,labelfont=sf,textfont=sf]{subfig}

\usepackage{booktabs}
 
\usepackage{url}

\begin{document}

\title{Revisiting Atomic Norm Minimization: A Sequential Approach for Atom Identification and Refinement}

\author{Xiaozhi Liu, Jinjiang Wei, Yong Xia
	% <-this % stops a space
	\thanks{Xiaozhi Liu and Jinjiang Wei contributed equally to this work.}
	\thanks{{Xiaozhi Liu, Jinjiang Wei and Yong Xia are with the School of Mathematical Sciences, Beihang University, Beijing 100191, China (xzliu@buaa.edu.cn; weijinjiang1999@163.com; yxia@buaa.edu.cn).}}
	\thanks{This work was supported by the National Key R\&D Program of China (Grant No. 2021YFA1003300).}
}

\maketitle

\begin{abstract}

Atomic norm minimization (ANM) is a key approach for line spectral estimation (LSE). Most related algorithms formulate ANM as a semidefinite programming (SDP), which incurs high computational cost. In this letter, we revisit the ANM problem and present a novel limit-based formulation, which dissects the essential components of the semidefinite characterization of ANM. Our new formulation does not depend on SDP and can be extended to handle more general atomic sets beyond mixture of complex sinusoids. Furthermore, we reveal the connection between ANM and Bayesian LSE approaches, bridging the gap between these two methodologies. Based on this new formulation, we propose a low-complexity algorithm called Sequential Atom Identification and Refinement (SAIR) for ANM. Simulation results demonstrate that SAIR achieves superior estimation accuracy and computational efficiency compared to other state-of-the-art methods.

\end{abstract}


\IEEEpeerreviewmaketitle

\section{Introduction} \label{sec:intro}

\IEEEPARstart{L}{ine} spectral estimation (LSE) is a fundamental technique in signal processing, with broad applications in direction-of-arrival (DOA) estimation \cite{l1-svd}, wireless communications \cite{ce_app}, and modal analysis \cite{modal}.
Given its importance across these domains, extensive research has focused on improving LSE's efficiency and accuracy.

Classical subspace methods, such as multiple signal classification (MUSIC) \cite{music},
address the LSE problem by exploiting the low-rank structure of the correlation matrix. 
However, these require a sufficient number of snapshots to accurately estimate the correlation matrix and assume the availability of complete data, which significantly limit their applicability in real-world scenarios.
With the advancement of compressive sensing (CS) \cite{cs}, many on-grid methods \cite{l1-svd} have been developed, which assume the true frequencies lie on a fixed sampling grid. However, LSE is a continuous parameter estimation problem, which inevitably leads to grid mismatch \cite{chi-mismatch}.

To tackle the issue of grid mismatch, several gridless methods have been proposed, including atomic norm minimization (ANM) approach \cite{anm} and enhanced matrix completion (EMaC) method \cite{emac}. 
ANM, in particular, solves the LSE problem using a continuous dictionary and provides strong theoretical guarantees in noiseless scenarios \cite{anm}.
However, these methods typically rely on solving semidefinite programming (SDP), resulting in high computational complexity.

Another class of methods employs an off-grid strategy to solve the LSE problem.
\cite{srcs} introduced an iterative reweighted approach based on the majorization-minimization (MM) algorithm, called Super-Resolution Compressed Sensing (SRCS).
\cite{nomp} extended the classical Orthogonal Matching Pursuit (OMP) to continuous parameter estimation using Newton refinements, resulting in Newtonized OMP (NOMP) algorithm. 
Furthermore, several Bayesian methods have been developed for solving the LSE problem \cite{sbl-de,vbi-17,superfast-lse}, utilizing Bayesian inference to estimate parameters with sparse priors.

This letter introduces a novel limit-based formulation of ANM which entirely bypasses the SDP implementation and can be extended to general atomic sets.
Furthermore, we reveal the connection between ANM and Bayesian approaches, offering new insights into the relationship between these two methodologies.
Based on this new formulation, we propose an efficient algorithm for ANM, called Sequential Atom Identification and Refinement (SAIR). Our simulations confirm that SAIR notably surpasses existing algorithms in both estimation accuracy and computational efficiency.

\textit{Notations}:
$\boldsymbol{A}$ is a matrix, $\boldsymbol{a}$ a vector, and $a$ a scalar.
$\boldsymbol{A}^T$ and $\boldsymbol{A}^H$ are the transpose and conjugate transpose of $\boldsymbol{A}$.
$\operatorname{rank}(\boldsymbol{A})$, $\operatorname{tr}(\boldsymbol{A})$, $\boldsymbol{A}^{-1}$, and $\boldsymbol{A}^{\dagger}$ represent the rank, trace, inverse, and pseudo-inverse of $\boldsymbol{A}$, respectively.
$\mathcal{R}(\boldsymbol{A})$ denotes the column space of $\boldsymbol{A}$, and $\boldsymbol{A} \succeq 0$ indicates that $\boldsymbol{A}$ is positive semidefinite (PSD).
$\|\boldsymbol{a}\|$ is the $\ell_2$-norm of $\boldsymbol{a}$.
$\operatorname{diag}(\boldsymbol{a})$ denotes a diagonal matrix with the elements of $\boldsymbol{a}$ on its diagonal.

\section{Preliminaries} \label{sec:pre}

Consider a LSE problem where the signal of interest $\boldsymbol{x} \in \mathbb{C}^n$ is composed of $K$ sinusoids:
\begin{equation}
	\boldsymbol{x} = \sum_{k=1}^{K} s_k \boldsymbol{a}(f_k),
\end{equation}
where $s_k \in \mathbb{C}$ represents the complex gain, and $f_k \in [0,1)$ denotes the frequency of the $k$-th component.
The atom $\boldsymbol{a}(f)$ follows a Vandermonde structure, given by:
\begin{equation}
	\boldsymbol{a}\left(f\right) = \left[1\ e^{-2 \pi i f}\ e^{-2 \pi i 2 f} \cdots e^{-2 \pi i(n-1) f} \right]^T.
	\label{eq:dft_basis}
\end{equation}
In the observation model, the observation vector $\boldsymbol{y} \in \mathbb{C}^m$ is defined as:
\begin{equation}
	\boldsymbol{y} = \boldsymbol{\Phi} \boldsymbol{x} + \boldsymbol{n},
	\label{eq:lse}
\end{equation}
where $\boldsymbol{\Phi} \in \mathbb{C}^{m\times n}$ (with $m \leq n$) is the measurement matrix\footnote{The matrix $\boldsymbol{\Phi}$ can either be the identity matrix (for the complete data case, $m=n$) or a random subsampling matrix formed by selecting $m$ random rows from the identity matrix (for the incomplete data case, $m < n$).}.
Additionally, $\boldsymbol{n}$ represents a white Gaussian noise vector with component variance $\sigma^2$.

The goal of LSE is to recover frequencies $\{f_k\}_{k=1}^K$ and gains $\{s_k\}_{k=1}^K$ from the measurement $\boldsymbol{y}$, without prior knowledge of the number of components $K$.

The LSE problem can be viewed as a continuous-domain parameter estimation task.
Due to its inherent sparsity, ANM \cite{anm} provides an effective solution.
Consider the atomic set $\mathcal{A}:=\left\{\boldsymbol{a}_k\right\}$, where each atom $\boldsymbol{a}_k \in \mathbb{C}^n$.
The atomic norm of $\boldsymbol{x}$ is defined as \cite{2012convex}:
\begin{equation}
	\|\boldsymbol{x}\|_{\mathcal{A}}=\inf _{\substack{\boldsymbol{a}_k, \\s_k \in \mathbb{C}, {r} }}\left\{\sum_{k=1}^{{r}}\left|s_k \right|: \boldsymbol{x}=\sum_{k=1}^{{r}} s_k \boldsymbol{a}_k, \boldsymbol{a}_k \in \mathcal{A} \right\} .
	\label{eq:an_general}
\end{equation}
In our problem, the atoms follow a specific structure: $\mathcal{A}:=\left\{\boldsymbol{a}\left(f_k\right), f_k \in[0,1)\right\}$, where $\boldsymbol{a}\left(f\right)$ is defined in \eqref{eq:dft_basis}.
In this case, the atomic norm can be characterized using an SDP \cite{anm}:
\begin{equation}
	\|\boldsymbol{x}\|_{\mathcal{A}}=\inf _{\boldsymbol{u}, t}\left\{\frac{1}{2 n} \operatorname{tr}(\mathcal{T}(\boldsymbol{u}))+\frac{1}{2} t\right\} \text {, s.t. }\left(\begin{array}{cc}
		\mathcal{T}(\boldsymbol{u}) & {\boldsymbol{x}} \\
		\boldsymbol{x}^H & t
	\end{array}\right) \succeq 0,
	\label{eq:an_sdp}
\end{equation}
where $\mathcal{T}(\boldsymbol{u})$ is a Toeplitz matrix with $\boldsymbol{u}$ as its first column.

Although ANM achieves precise frequency localization under specific minimum separation conditions \cite{anm}, its reliance on SDP results in high computational complexity, limiting scalability in large-scale applications.

Furthermore, in noiseless compressive measurement scenarios, if the signal $\boldsymbol{x}$ has an atomic representation $\boldsymbol{x} = \sum_{j=1}^{r} \boldsymbol{a}(f_j)$, then $\boldsymbol{\Phi} \boldsymbol{x} = \sum_{j=1}^{r} \boldsymbol{\Phi} \boldsymbol{a}(f_j)$.
This indicates that, within the truncated atomic set $\mathcal{A} := \left\{\boldsymbol{\Phi} \boldsymbol{a}\left(f_k\right), f_k \in[0,1)\right\}$, the sparsity of $\boldsymbol{x}$ is preserved in $\boldsymbol{\Phi} \boldsymbol{x}$.
In other words, we can recover the sparse atomic representation of $\boldsymbol{x}$ by directly solving the truncated atomic norm $\|\boldsymbol{\Phi} \boldsymbol{x}\|_{\mathcal{A}}$.

While no equivalent SDP formulation exists for solving the truncated atomic norm, we address this by developing a novel formulation for the atomic norm over a general atomic set, eliminating the need for an SDP.
This enables us to propose a low-complexity algorithm that efficiently recovers frequencies.

\section{Equivalent Reformulation of Atomic Norm} \label{sec:equivalent}

Before presenting the new formulation of ANM, we first review some fundamental results that underpin the SDP characterization.
\begin{lem}[Schur Complement] \label{lem:schur}
	Let $\boldsymbol{M}=\left(\begin{array}{cc}\boldsymbol{A} & \boldsymbol{B} \\ \boldsymbol{B}^{H} & \boldsymbol{D}\end{array}\right)$ be a Hermitian matrix. The following conditions are equivalent:
	\begin{itemize}
		\item[(1)] $\boldsymbol{M} \succeq 0$.
		\item[(2)] $\boldsymbol{A} \succeq 0$, $(\boldsymbol{I} - \boldsymbol{A}\boldsymbol{A}^{\dagger})\boldsymbol{B} = 0$, $\boldsymbol{D} - \boldsymbol{B}^H\boldsymbol{A}^{\dagger}\boldsymbol{B} \succeq 0$.
	\end{itemize}
\end{lem}
\begin{remk}
	The condition $(\boldsymbol{I} - \boldsymbol{A}\boldsymbol{A}^{\dagger})\boldsymbol{B} = 0$ can be equivalently expressed as $\boldsymbol{B} \in {\mathcal{R}(\boldsymbol{A})}$ \cite{boyd2004convex}.
\end{remk}

\begin{lem}[Vandermonde Decomposition \cite{toeplitz1911theorie}] \label{lem:vander}
	A Toeplitz matrix $\mathcal{T}(\boldsymbol{u})$ can be decomposed as:
	\begin{equation}
		\mathcal{T}(\boldsymbol{u}) = \boldsymbol{V}\boldsymbol{D}\boldsymbol{V}^H = \sum_{j=1}^{r} d_j \boldsymbol{a}(f_j) \boldsymbol{a}(f_j)^H,
	\end{equation}
	where $r = \operatorname{rank} (\mathcal{T}(\boldsymbol{u}))$, $\boldsymbol{V} = [\boldsymbol{a}(f_1), \boldsymbol{a}(f_2),\cdots, \boldsymbol{a}(f_r)]$, $f_j \in [0,1)$, and $\boldsymbol{D}=\operatorname{diag}(d_1,d_2,\cdots,d_r)$ with $d_j>0$.
\end{lem}

\begin{lem} \label{lem:limit}
	Let $\boldsymbol{T} \succeq 0$ with $\operatorname{{rank}}(\boldsymbol{T}) = r$. If $\boldsymbol{x} \in \mathcal{R}(\boldsymbol{T})$, then:
	\begin{equation}
		\lim_{\beta \rightarrow 0 } \boldsymbol{x}^H (\boldsymbol{T} + \beta \boldsymbol{I} )^{-1} \boldsymbol{x} = \boldsymbol{x}^H (\boldsymbol{T} )^{\dagger} \boldsymbol{x}.
	\end{equation}
	Otherwise:
	\begin{equation}
		\lim_{\beta \rightarrow 0 } \boldsymbol{x}^H (\boldsymbol{T} + \beta \boldsymbol{I} )^{-1} \boldsymbol{x} = +\infty.
	\end{equation}
\end{lem}
This lemma can be proven using eigenvalue decomposition, though the detailed proof is omitted here.
The condition $\boldsymbol{x} \in \mathcal{R}(\boldsymbol{T})$ is challenging to handle directly in optimization problems. Lemma \ref{lem:limit} reformulates this condition in a limit-based form, facilitating algorithm design.

Based on these key properties of low-rank PSD Toeplitz matrices, we establish an alternative formulation of the atomic norm's semidefinite characterization.
\begin{thm} \label{thm:limit_sdp}
	For any mixture of complex sinusoids $\boldsymbol{x} \in \mathbb{C}^n$, the atomic norm can be expressed as:
	\begin{equation}
		\|\boldsymbol{x}\|_{\mathcal{A}} = \inf _{\substack{d_j >0, \\f_j \in [0,1), {r}}} \lim_{\beta \rightarrow 0} \left\{\frac{1}{2} \sum_{j=1}^{r} d_j +\frac{1}{2} \boldsymbol{x}^H \boldsymbol{C}^{-1} \boldsymbol{x}\right\},
		\label{eq:an_limit}
	\end{equation}
	where $\boldsymbol{C} = \sum_{j=1}^{r} d_j \boldsymbol{a}(f_j)\boldsymbol{a}(f_j)^H + \beta \boldsymbol{I}$.
\end{thm}
\begin{proof}
	By Lemma \ref{lem:schur}, the semidefinite characterization of the atomic norm in \eqref{eq:an_sdp} can be reformulated as:
	\begin{equation}
		\begin{aligned}
			\|\boldsymbol{x}\|_{\mathcal{A}} = &\inf _{\boldsymbol{u}}\left\{\frac{1}{2 n} \operatorname{tr}(\mathcal{T}(\boldsymbol{u}))+\frac{1}{2} \boldsymbol{x}^H\mathcal{T}(\boldsymbol{u})^{\dagger}\boldsymbol{x}\right\}, \\  &\text {s.t. } \mathcal{T}(\boldsymbol{u}) \succeq 0, \boldsymbol{x} \in \mathcal{R}(\mathcal{T}(\boldsymbol{u})).
		\end{aligned}
	\end{equation}
	Using the Vandermonde decomposition of the PSD Toeplitz matrices (see Lemma \ref{lem:vander}), we have
	\begin{equation}
		\begin{aligned}
			\|\boldsymbol{x}\|_{\mathcal{A}} = &\inf _{\substack{d_j >0, \\f_j \in [0,1), {r}}}\left\{\frac{1}{2} \sum_{j=1}^{r} d_j + \frac{1}{2} \boldsymbol{x}^H(\sum_{j=1}^{r}d_j \boldsymbol{a}(f_j)\boldsymbol{a}(f_j)^H)^{\dagger}\boldsymbol{x}\right\}, \\  &\text {s.t. } \boldsymbol{x} \in \mathcal{R}(\sum_{j=1}^{r}d_j \boldsymbol{a}(f_j)\boldsymbol{a}(f_j)^H).
		\end{aligned}
		\label{eq:an_van}
	\end{equation}
	Finally, leveraging the result of Lemma \ref{lem:limit}, we can establish that the right-hand side of equation \eqref{eq:an_limit} is equivalent to that of equation \eqref{eq:an_van}. 
\end{proof}

The above results provide a limit-based formulation of atomic norm for mixed sinusoids from its semidefinite characterization.
In fact, this new formulation can be extended to more general atomic sets where SDP may not be established.

\begin{thm} \label{thm:limit_general_atom}
	Consider the atomic set $\mathcal{A}:=\left\{\boldsymbol{a}_k\right\}$.
	Assume that any $r \leq n$ distinct atoms $\{\boldsymbol{a}_1, \boldsymbol{a}_2, \cdots, \boldsymbol{a}_r\}$ in the set are linearly independent. 
	Then, the following holds:
	\begin{equation}
		\begin{aligned}
			\|\boldsymbol{x}\|_{\mathcal{A}} = &\inf _{\substack{d_j >0, \\\boldsymbol{a}_j \in \mathcal{A}, {r}}}\left\{\frac{1}{2} \sum_{j=1}^{r} d_j + \frac{1}{2} \boldsymbol{x}^H(\sum_{j=1}^{r}d_j \boldsymbol{a}_j\boldsymbol{a}_j^H)^{\dagger}\boldsymbol{x}\right\}, \\  &\text{s.t. } \boldsymbol{x} \in \mathcal{R}(\sum_{j=1}^{r}d_j \boldsymbol{a}_j\boldsymbol{a}_j^H)
		\end{aligned}
		\label{eq:an_limit_general}
	\end{equation}
\end{thm}
\begin{proof}
	Let the right-hand side of \eqref{eq:an_limit_general} be denoted as $\operatorname{Limit}(\boldsymbol{x})$. 
	Suppose $\boldsymbol{x} = \sum_{j=1}^{r} s_j \boldsymbol{a}_j = \boldsymbol{A} \boldsymbol{s}$, with $\boldsymbol{A} = \left[\boldsymbol{a}_1,\boldsymbol{a}_2,\cdots,\boldsymbol{a}_r\right]$ where $r\leq n$, and $\boldsymbol{s} = \left[s_1,s_2,\cdots,s_r\right]$, where $s_j \neq 0$.
	Let $\sum_{j=1}^{r}\left|s_j\right| \boldsymbol{a}_j\boldsymbol{a}_j^H = \boldsymbol{A}\boldsymbol{D}\boldsymbol{A}^H$, where $\boldsymbol{D} = \operatorname{diag}(\left|s_1\right|,\left|s_2\right|,\cdots,\left|s_r\right|)$. 
	By assumption, $\operatorname{{rank}}(\sum_{j=1}^{r}\left|s_j\right| \boldsymbol{a}_j\boldsymbol{a}_j^H) = \operatorname{rank}(\boldsymbol{A}) = r$.
	Therefore, we have:
	\begin{equation}
		\begin{aligned}
			&\frac{1}{2} \sum_{j=1}^{r} \left|s_j\right| + \frac{1}{2} \boldsymbol{x}^H(\sum_{j=1}^{r}\left|s_j\right| \boldsymbol{a}_j\boldsymbol{a}_j^H)^{\dagger}\boldsymbol{x} \\
			&= \frac{1}{2} \sum_{j=1}^{r} \left|s_j\right| + \frac{1}{2} \boldsymbol{x}^H (\boldsymbol{A}\boldsymbol{D}\boldsymbol{A}^H)^{\dagger}\boldsymbol{x} \\
			&= \frac{1}{2} \sum_{j=1}^{r} \left|s_j\right| + \frac{1}{2} \sum_{j=1}^{r} \frac{\left|s_j\right|^2}{\left|s_j\right|} = \sum_{j=1}^{r} \left|s_j\right|.
		\end{aligned}
	\end{equation}
	The second equality follows from the full rank decomposition and Moore-Penrose pseudoinverse. Specifically, for any Hermitian matrix $\boldsymbol{M}$ with full-rank decomposition $\boldsymbol{M} = \boldsymbol{B} \boldsymbol{B}^H$, we have $(\boldsymbol{M})^{\dagger} = \boldsymbol{B} (\boldsymbol{B}^H \boldsymbol{B})^{-1} (\boldsymbol{B}^H \boldsymbol{B})^{-1} \boldsymbol{B}^H$.
	Since this holds for any decomposition of $\boldsymbol{x}$, it follows that $\|\boldsymbol{x}\|_{\mathcal{A}}  \geq \operatorname{Limit} (\boldsymbol{x})$.
	
	Conversely, suppose for some $\boldsymbol{a}_j$, $\boldsymbol{x} = \sum_{j=1}^{r} s_j \boldsymbol{a}_j$. Then,
	\begin{equation}
		\begin{aligned}
			&\frac{1}{2} \sum_{j=1}^{r} d_j + \frac{1}{2} \boldsymbol{x}^H(\sum_{j=1}^{r}d_j \boldsymbol{a}_j\boldsymbol{a}_j^H)^{\dagger}\boldsymbol{x} \\
			&= \frac{1}{2} \sum_{j=1}^{r} d_j + \frac{1}{2} \sum_{j=1}^{r} \frac{\left|s_j\right|^2}{d_j} \geq \sum_{j=1}^{r} \left|s_j\right| \geq \|\boldsymbol{x}\|_{\mathcal{A}}.
		\end{aligned}
	\end{equation}
	This implies that $\operatorname{Limit} (\boldsymbol{x}) \geq \|\boldsymbol{x}\|_{\mathcal{A}}$, as the inequalities hold for any feasible $d_j$ and $\boldsymbol{a}_j$.
\end{proof}

Using Theorem \ref{thm:limit_general_atom} and Lemma \ref{lem:limit}, we present a limit-based formulation for the atomic norm with respect to a general atomic set $\mathcal{A}$:
\begin{equation}
	\|\boldsymbol{x}\|_{\mathcal{A}} = \inf_{\boldsymbol{\theta}}  \lim_{\beta \rightarrow 0} \mathcal{L}(\boldsymbol{\theta},{\beta}),
	\label{eq:inf_limit}
\end{equation}
where the objective function $\mathcal{L}(\boldsymbol{\theta},{\beta})$ is defined as
\begin{equation}
	\mathcal{L}(\boldsymbol{\theta},{\beta}) = \frac{1}{2} \sum_{j=1}^{r} d_j +\frac{1}{2} \boldsymbol{x}^H(\sum_{j=1}^{r} d_j \boldsymbol{a}_j\boldsymbol{a}_j^H + \beta \boldsymbol{I})^{-1} \boldsymbol{x},
\end{equation}
and {$\boldsymbol{\theta} = \{\boldsymbol{a}_j,d_j,j=1,\cdots,r\}$} with $d_j > 0$.

Theorem \ref{thm:limit_general_atom} demonstrates that this limit-based formulation applies not only to atomic sets of equispaced sampled sinusoids but also to other sets, such as the orthogonal basis $\left\{\boldsymbol{e}_i\right\}_{i=1}^n$ and randomly sampled sinusoids (i.e., the truncated atomic set mentioned earlier).

However, solving the problem in \eqref{eq:inf_limit} remains challenging. To our knowledge, there is no established solution for this limit-based form with a general atomic set $\mathcal{A}$.
For this, we highlight the following finding: the interchangeability of the limit and infimum operations in \eqref{eq:inf_limit}.
This result enables us to reformulate the problem as a series of minimization problems over $\beta$, laying the foundation for our algorithm.

\begin{thm}
	Let $\mathcal{A} = \left\{\boldsymbol{a}_k\right\}$ be an atomic set, where any subset of $r \leq n$ distinct atoms is assumed to be linearly independent. Under this condition, the following equalities hold:
	\begin{equation}
		\|\boldsymbol{x}\|_{\mathcal{A}} = \inf_{\boldsymbol{\theta}}  \lim_{\beta \rightarrow 0} \mathcal{L}(\boldsymbol{\theta},{\beta}) = \lim_{\beta \rightarrow 0} \inf_{\boldsymbol{\theta}} \mathcal{L}(\boldsymbol{\theta},{\beta}).
		\label{eq:lim_inf}
	\end{equation}
\end{thm}

\begin{proof}
	By Theorem \ref{thm:limit_general_atom}, we have
	\begin{equation}
		\|\boldsymbol{x}\|_{\mathcal{A}} = \inf_{\boldsymbol{\theta}} \lim_{\beta \rightarrow 0} \mathcal{L}(\boldsymbol{\theta},{\beta}) = \lim_{\beta \rightarrow 0} \mathcal{L}(\boldsymbol{\theta}^{*},{\beta}) \geq \varliminf_{\beta \rightarrow 0} \inf_{\boldsymbol{\theta}} \mathcal{L}(\boldsymbol{\theta},{\beta}),
	\end{equation}
	where $	\boldsymbol{\theta}^{*} = \arg \inf_{\boldsymbol{\theta}} \lim_{\beta \rightarrow 0} \mathcal{L}(\boldsymbol{\theta},{\beta})$.
	We assert that the inequality above is actually an equality, indicating that the limit exists. 
	To prove this by contradiction, assume
	\begin{equation}
		\varliminf_{\beta \rightarrow 0} \inf_{\boldsymbol{\theta}} \mathcal{L}(\boldsymbol{\theta},{\beta}) = \|\boldsymbol{x}\|_{\mathcal{A}} - 3l, l > 0.
	\end{equation}
	This implies that there exists a sequence $\{\beta_n\}$ and an integer $n_1$ such that for all $n \geq n_1$
	\begin{equation}
		\mathcal{L}(\boldsymbol{\theta}_{\beta_n}^{*},{\beta_n}) \leq \|\boldsymbol{x}\|_{\mathcal{A}} - 2l,
		\label{eq:assume_ineq}
	\end{equation}
	where $\boldsymbol{\theta}_{\beta_n}^{*} = \arg \inf_{\boldsymbol{\theta}} \mathcal{L}(\boldsymbol{\theta},{\beta_n})$.
	By Theorem \ref{thm:limit_general_atom}, we have
	\begin{equation}
		\lim_{\beta \rightarrow 0} \mathcal{L}(\boldsymbol{\theta},{\beta}) \geq \|\boldsymbol{x}\|_{\mathcal{A}}, \forall \boldsymbol{\theta}.
	\end{equation}
	Thus, for any $\boldsymbol{\theta}$, there exists an integer $n_2$ such that for all $n \geq n_2$,
	\begin{equation}
		\mathcal{L}(\boldsymbol{\theta},{\beta_n}) \geq \|\boldsymbol{x}\|_{\mathcal{A}} - l.
	\end{equation}
	This contradicts the inequality in \eqref{eq:assume_ineq}, thus proving our initial claim that \eqref{eq:lim_inf} hold.
\end{proof}

\begin{figure*}[t]
	\centering
	\subfloat[]{\includegraphics[width=2in]{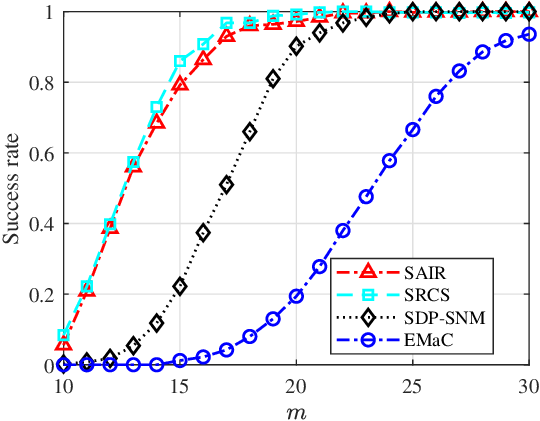}%
		\label{fig:sr_m}}
	\hfil
	\subfloat[]{\includegraphics[width=2in]{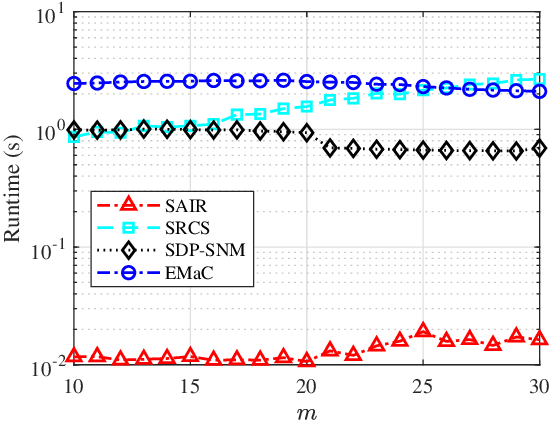}%
		\label{fig:runtime_m}}
	\caption{Simulation results of respective algorithms when $K=5$. (a) Success rates vs. $m$. (b) Runtimes vs. $m$.}
	\label{fig:sr_runtime_m}
\end{figure*}

\section{Connection to Prior Arts} \label{sec:bayesian}

In this section, we will show that the new atomic norm formulation reveals the relationship between ANM and Bayesian LSE approaches, bridging the gap between these two methodologies.
Previous works \cite{sbl-de, vbi-17, superfast-lse} have adopted a Bayesian perspective for LSE, formulating the objective of maximum a posteriori (MAP) as (see \cite{superfast-lse} for further detail):
\begin{equation}
	\mathcal{L}(\boldsymbol{z},\zeta,\sigma,\boldsymbol{f},\boldsymbol{\gamma}) = \ln |\mathbf{C}|+\boldsymbol{y}^{\mathrm{H}} \mathbf{C}^{-1} \boldsymbol{y} + \operatorname{const},
	\label{eq:obj_bayesian}
\end{equation}
where the covariance matrix is:
$$
\boldsymbol{C} = \sum_{j=1}^{n} \gamma_j \boldsymbol{\Phi} \boldsymbol{a}(f_j) (\boldsymbol{\Phi}\boldsymbol{a}(f_j))^H + \sigma^2 \boldsymbol{I}.
$$
This objective parallels the new atomic norm formulation in  \eqref{eq:an_limit}, where the noise variance $\sigma^2$ provides a statistical interpretation for the parameter $\beta$.
Additionally, from the perspective of the MM algorithm \cite{yang_ram}, and given that the function $f(\boldsymbol{M}) = \ln |\boldsymbol{M} + \sigma^2 \boldsymbol{I}|$ has a tangent plane at the origin given by
$$
\ln |\sigma^2 \boldsymbol{I}| + \frac{1}{\sigma^2} \operatorname{tr} (\boldsymbol{M}),
$$
the atomic norm in \eqref{eq:an_limit} can be interpreted as the first-order Taylor expansion of the objective’s first term in \eqref{eq:obj_bayesian} at the origin, as the noise parameter $\sigma \rightarrow 0$.

Furthermore, the objective \eqref{eq:obj_bayesian} employs the log-det sparse metric, which serves as a non-convex relaxation of the atomic $\ell_0$-norm \cite{yang_ram}:
\begin{equation}
	\|\boldsymbol{x}\|_{\mathcal{A},0}=\inf _{\substack{f_k \in [0,1), \\s_k \in \mathbb{C}, {r} }}\left\{{r}: \boldsymbol{x}=\sum_{k=1}^{{r}} s_k \boldsymbol{a}\left(f_k\right), \boldsymbol{a}\left(f_k\right) \in \mathcal{A} \right\} .
	\label{eq:an_0}
\end{equation}
In contrast, the atomic norm in \eqref{eq:an_general} is a convex relaxation.

\section{SAIR Algorithm} \label{sec:algorithm}

In this section, we introduce the Sequential Atom Identification and Refinement (SAIR) algorithm based on the limit-based formulation of the atomic norm.
Leveraging its connection to Bayesian approaches, we adopt an off-grid strategy similar to that of \cite{superfast-lse}.
The objective function is defined as:
\begin{equation}
	\mathcal{L}({\boldsymbol{f},\boldsymbol{d}},{\beta_k}) = \frac{1}{2} \sum_{j=1}^{r} d_j +\frac{1}{2} \boldsymbol{x}^H\boldsymbol{C}^{-1} \boldsymbol{x},
\end{equation}
where $\boldsymbol{f} = [f_1, \dots, f_r]$ represents the frequencies in the current recovery dictionary, and $\boldsymbol{d} = [d_1, \dots, d_r]$ denotes the magnitudes.
The matrix $\boldsymbol{C}$ is given by:
$$
\boldsymbol{C} = \sum_{j=1}^{r} d_j \boldsymbol{a}(f_j) \boldsymbol{a}(f_j)^H + \beta_k \boldsymbol{I},
$$
where $\{\beta_k\}$ is a sequence approaching zero.

Now, we explore how the objective function changes when a new atom is added at frequency $f$ from a finite discrete set $\boldsymbol{\Omega} = \{k/(\gamma n): k = 1,\cdots,(\gamma n -1)\}$, where $\gamma$ is the oversampling factor.
Using the Woodbury's matrix inversion identity \cite{tipping_fast}, the change in the objective is given by:
\begin{equation}
	\Delta \mathcal{L}(f,d) = \frac{1}{2} (d - \frac{1}{d^{-1}+\boldsymbol{a}(f)^H\boldsymbol{C}^{-1}\boldsymbol{a}(f)} |\boldsymbol{a}(f)^H\boldsymbol{C}^{-1}\boldsymbol{x}|^2).
\end{equation}
For a fixed frequency $f$, the optimal magnitude $\hat{d}(f)$ is obtained by minimizing $\Delta \mathcal{L}(f,d)$.
If $\Delta \mathcal{L}(\bar{f}_j,\hat{d}(\bar{f}_j)) \geq 0$ for all candidate frequencies $\bar{f}_j \in \boldsymbol{\Omega}$, the algorithm terminates.
Otherwise, the next atom is selected by minimizing:
\begin{equation}
	\tilde{f} = \arg \min_{{f\in \boldsymbol{\Omega}}} \Delta \mathcal{L}(f,\hat{d}(f)).
\end{equation}
To mitigate the effects of grid mismatch, we employ the damped Broyden-Fletcher-Goldfarb-Shanno (BFGS) algorithm \cite{numerical_opt} to refine frequencies in the dictionary over a continuous domain.

The algorithm starts with an empty dictionary and initializes $\boldsymbol{C} = \beta_0 \boldsymbol{I}$, where $\beta_0 = \frac{1}{n}\|\boldsymbol{x}\|$.
The sequence is then updated as $\beta_{k+1} = 0.2 \beta_k$. 
The oversampling factor, $\gamma$, is set to $8$.
The SAIR algorithm can be easily extended to compressive scenarios, where it corresponds to considering the truncated atomic set $\mathcal{A} := \left\{\boldsymbol{\Phi} \boldsymbol{a}\left(f_k\right), f_k \in[0,1)\right\}$.
To reduce computational complexity, derivative calculations and matrix inversion of $\boldsymbol{C}$ can be accelerated following \cite{superfast-lse}.
Although a convergence analysis similar to that in \cite{superfast-lse} could be applied to the SAIR algorithm, it lies beyond the scope of this letter. In all tested cases, the algorithm converged consistently, demonstrating strong stability.

\section{Simulation Results} \label{sec:simulation}

Simulations are performed to illustrate the performance of our proposed algorithm for DOA estimation in a noiseless scenario.
We consider a mixture of $K$ sinusoids of length $n=64$, where the frequencies and coefficients are randomly generated, ensuring the minimum frequency separation $\Delta f_{\operatorname{min}} = \min_{k\neq j} |f_k - f_j| \geq 2/n$.
Two metrics are used: success rate and runtime.
The success rate is the proportion of trials in which the signal is successfully recovered.
A trial is considered successful if the normalized mean squared error (NMSE) satisfies $\operatorname{NMSE} \leq 10^{-4}$, where NMSE is defined as
$$
\operatorname{NMSE} := \frac{\|\hat{\boldsymbol{x}}- \boldsymbol{x}\|^2}{\|\boldsymbol{x}\|^2}.
$$
We perform 500 trials for each scenario.
We compare SAIR with SDP-ANM \cite{anm}, EMaC \cite{emac} and SRCS \cite{srcs}.

In Fig. \ref{fig:sr_m}, we illustrate the success rates as a function of the number of measurements $m$, where SAIR and SRCS achieve the highest rates, outperforming SDP-ANM and EMaC.
As $m$ decreases, SAIR’s advantage over SDP-ANM becomes more evident.

Fig. \ref{fig:runtime_m} presents the runtimes across various values of $m$, showing that SAIR achieves significantly lower runtimes, improving by over two orders of magnitude compared to other methods.
This highlights the superior computational efficiency of the SAIR algorithm.

\section{Conclusion} \label{sec:conclusion}

This letter presents an alternative formulation of the atomic norm's SDP characterization using a limit-based approach.
This new formulation leverages the intrinsic structure of the atomic norm and enables broader applicability to general atomic sets, opening new avenues for future research.
By exploiting the connection between this new formulation and Bayesian approaches, we propose a fast SAIR algorithm for ANM, which operates in an off-grid manner to solve sequential subproblems. 
Simulation results validate its superior estimation accuracy and computational efficiency compared to existing methods in DOA estimation.

\bibliographystyle{IEEEtran}
\bibliography{IEEEabrv,mybib}

\end{document}